\pdfoutput=1
\documentclass[11pt]{article}

\usepackage[T1]{fontenc}
\usepackage[margin=1in]{geometry}
\usepackage{amsmath,amssymb,amsthm,mathtools,bm}
\usepackage{graphicx}
\usepackage{microtype}
\usepackage{xurl}
\usepackage{enumitem}
\usepackage[colorlinks=true,linkcolor=blue,citecolor=blue,urlcolor=blue]{hyperref}

\newtheorem{theorem}{Theorem}[section]
\newtheorem{proposition}[theorem]{Proposition}
\newtheorem{lemma}[theorem]{Lemma}
\newtheorem{corollary}[theorem]{Corollary}
\newtheorem{assumption}[theorem]{Assumption}
\theoremstyle{definition}
\newtheorem{definition}[theorem]{Definition}

\newcommand{\keywords}[1]{\par\smallskip\noindent\textbf{Keywords.} #1}
\newcommand{\subclass}[1]{\par\smallskip\noindent\textbf{Mathematics Subject Classification.} #1}

\newcommand{\Prob}{\mathbb P}
\newcommand{\E}{\mathbb E}
\newcommand{\R}{\mathbb R}
\newcommand{\1}{\mathbf 1}
\newcommand{\CP}{\mathcal{CP}}
\newcommand{\DNN}{\mathcal{DNN}}
\newcommand{\diag}{\operatorname{diag}}
\newcommand{\supp}{\operatorname{supp}}

\begin{document}

\title{Exactness of the DNN Relaxation for Random Standard Quadratic Programs}

\author{Xin Chen\\[0.3em]
\small H. Milton Stewart School of Industrial and Systems Engineering\\
\small Georgia Institute of Technology, Atlanta, GA, USA\\
\small \texttt{xin.chen@isye.gatech.edu}\\
\small ORCID: \url{https://orcid.org/0000-0002-5168-4823}}

\date{}

\maketitle

\begin{abstract}
We study the doubly nonnegative (DNN) relaxation of the standard quadratic optimization problem
\[
\min\{x^\top Qx:\ x\in\Delta^{n-1}\},\qquad \Delta^{n-1}:=\{x\in\R_+^n:\ \1^\top x=1\},
\]
for random symmetric matrices with independent diagonal and off-diagonal entries. Let $m_n:=\min_{1\le i\le n} Q_{ii}$ and set $M:=Q-m_nE$, where $E$ is the all-ones matrix. The negative off-diagonal entries of $M$ define a defect graph $G_n^-$. Under entrywise independence, absolute continuity, and the tail-decay condition $n^5\E[F_O(m_n)^4]\to 0$, where $F_O$ is the off-diagonal distribution function, we prove that with probability tending to one every defect component has size at most $4$. On this event, the shifted DNN value decomposes over defect components. Since the DNN and completely positive cones coincide in dimensions at most four, each local relaxation is exact. A finite KKT-candidate argument gives local uniqueness, and absolute continuity rules out ties, so the global DNN optimizer is unique and rank one. The graph estimate uses the fact that every connected component of size at least five contains a tree on exactly five vertices. For Gaussian orthogonal ensemble data, we prove the explicit bound
\[
\Prob\bigl(\text{the DNN optimizer is unique and rank one}\bigr)
\ge 1-K\frac{(\ln n)^2}{n^3}.
\]
On the same event, the exact optimizer can be recovered in $O(n^2)$ time by constructing the defect graph and solving constant-size local KKT systems. We also verify the tail condition for variance-tuned Gaussian Wigner models, heavy-tailed laws, and finite-lower-endpoint laws.
\subclass{90C20 ; 90C22 ; 60B20 ; 05C80}
\end{abstract}

\section{Introduction}
The standard quadratic optimization problem (StQP)
\[
\min\{x^\top Qx:\ x\in\Delta^{n-1}\}
\]
is a classical nonconvex quadratic program with deep ties to copositive programming, graph formulations, and continuous reformulations of discrete optimization problems; see Bomze~\cite{Bomze1998} and Bomze and de Klerk~\cite{BomzeDeKlerk2002}. Its completely positive lifting is exact but computationally intractable in general. Replacing the completely positive cone with the doubly nonnegative cone leads to a tractable semidefinite relaxation. This paper investigates that relaxation in a random-data regime.

For random StQPs, the sparsity of optimal solutions was initiated by Chen, Peng, and Zhang~\cite{ChenPengZhang2013}, sharpened by Chen and Peng~\cite{ChenPeng2015}, and further developed by Chen and Pittel~\cite{ChenPittel2021}. Those papers analyze the support size of the primal optimizer. We focus instead on when the DNN relaxation itself is exact, when its optimizer is unique, and when that optimizer is rank one.

The analysis is based on the shifted matrix $M=Q-m_nE$, where $m_n:=\min_{1\le i\le n} Q_{ii}$. The entries $M_{ij}<0$ form a random defect graph $G_n^-$. Once this shift is made, the signs of the off-diagonal entries give a deterministic decomposition over the connected components of $G_n^-$. Under the tail assumptions below, the defect graph has no components of size at least five with probability tending to one. Since $\DNN^r=\CP^r$ for $r\le4$~\cite{Diananda,BermanShaked}, every local DNN problem on a component of size at most four is exact. It then remains to prove local uniqueness and to rule out ties between different components.

Our first main result establishes general conditions for exactness. Under entrywise independence, absolute continuity, and a four-edge tail-decay condition, the DNN relaxation has a unique rank-one optimizer with probability $1-o(1)$. Specializing this framework to GOE data, the second main result gives an explicit failure bound: the DNN relaxation is exact and has a unique rank-one optimizer with probability at least $1-K(\ln n)^2/n^3$. We also record extensions to variance-tuned Gaussian Wigner models, heavy-tailed laws, and finite-lower-endpoint laws. Since the input matrix has $N=n^2$ entries, the constructive proof naturally yields an $O(n^2)$ exact algorithm on the same event: one scans the matrix to build the defect graph, extracts its connected components, and solves only constant-size local KKT systems.

We highlight two features of the analysis. To avoid a naive sum over all connected component sizes $r\ge5$, which leads to a divergent factorial overcount, we observe that every connected component of size at least five contains a tree on exactly five vertices. Thus it is enough to count labeled five-vertex trees. Furthermore, the conclusion is stronger than a vertex-recovery statement: even when the original StQP minimizer is not a vertex, the DNN relaxation remains exact because the defect graph typically breaks into small components, naturally forcing the local DNN problem to be exact on each component.

Exactness of semidefinite and doubly nonnegative relaxations has also been studied in broader QCQP settings, including both deterministic conditions and random-instance results. For quadratically constrained quadratic programs (QCQPs), identifying structural conditions that guarantee exact semidefinite relaxations is a major theme; see, for example, Lavaei and Low~\cite{LavaeiLow2012}, Burer and Ye~\cite{BurerYe2020,BurerYeCorrection2021}, and Wang and K{\i}l{\i}n\c{c}-Karzan~\cite{WangKilincKarzan2021}. Burer and Ye, in particular, also consider random general QCQPs and high-probability exactness of the Shor relaxation. For the standard quadratic program specifically, the problem admits an exact convex conic formulation over the completely positive cone. Replacing this computationally intractable cone with the DNN cone yields a tractable relaxation; see Bomze and de Klerk~\cite{BomzeDeKlerk2002}. Because the CP and DNN cones coincide for dimensions up to four but diverge for dimensions five and higher~\cite{Diananda,BermanShaked,BurerAnstreicherDur2009}, exactness depends strongly on the absence of large connected substructures. Our work complements this literature by showing that, under the random models considered here, the defect graph typically decomposes into low-dimensional components on which the DNN and CP cones coincide.

\section{Universal model and main results}\label{sec:model}
We write $\R_+:=[0,\infty)$ and let $\1$ denote the all-ones vector, with dimension determined by context. For a positive integer $n$, set $[n]:=\{1,\ldots,n\}$; for a finite set $S$, $|S|$ denotes its cardinality. Set $E:=\1\1^\top$ for the all-ones matrix of appropriate dimension. For finite index sets $A$ and $B$, let $E_{A,B}$ denote the $|A|\times |B|$ all-ones matrix; when $A=B$, we write $E_A:=E_{A,A}$. For a generic matrix $X$, $X[A,B]$ denotes the submatrix with row indices $A$ and column indices $B$, and we abbreviate $X[A]:=X[A,A]$.

For sequences, we write $a_n\asymp b_n$ when there exist constants $0<k_1\le k_2<\infty$, independent of $n$, such that $k_1 b_n\le a_n\le k_2 b_n$ for all sufficiently large $n$. For continuous functions, $f(x)\asymp g(x)$ as $x\to x_0$ has the analogous meaning. We write $a_n\sim b_n$ when $a_n/b_n\to1$; for functions, $f(x)\sim g(x)$ as $x\to x_0$ has the analogous meaning.

For a symmetric matrix $Z\in\R^{r\times r}$, define
\[
\begin{aligned}
\vartheta(Z)&:=\min\{\langle Z,Y\rangle:\ \langle E,Y\rangle=1,\ Y\succeq0,\ Y\ge0\},\\
\mu(Z)&:=\min\{x^\top Zx:\ x\in\Delta^{r-1}\},
\end{aligned}
\]
where $\Delta^{r-1}:=\{x\in\R_+^r:\ \1^\top x=1\}$. Here $Y\succeq0$ means positive semidefinite, and $Y\ge0$ means entrywise nonnegative.

We consider the DNN relaxation
\begin{equation}\label{eq:dnnQ}
\vartheta(Q)=\min\{\langle Q,X\rangle:\ \langle E,X\rangle=1,\ X\succeq0,\ X\ge0\}
\end{equation}
of the standard quadratic program.

\begin{assumption}[Matrix generator]\label{ass:generator}
Let $Q=(Q_{ij})_{1\le i,j\le n}$ be a random symmetric matrix such that:
\begin{enumerate}[label=(A\arabic*)]
\item \emph{Entrywise independence.} The diagonal entries $(Q_{ii})_{i=1}^n$ and the upper-triangular off-diagonal entries $(Q_{ij})_{1\le i<j\le n}$ are mutually independent.
\item \emph{Absolute continuity.} The diagonal entries are i.i.d. with distribution function $F_D$, and the off-diagonal entries are i.i.d. with distribution function $F_O$. Both laws are absolutely continuous with respect to Lebesgue measure.
\item \emph{Tail-decay condition.} Writing
\[
 m_n:=\min_{1\le i\le n}Q_{ii},
\]
we have
\begin{equation}\label{eq:tailthreshold}
 n^5\,\E\bigl[F_O(m_n)^4\bigr]\to 0.
\end{equation}
\end{enumerate}
\end{assumption}

After shifting by $m_nE$, define
\begin{equation}\label{eq:def-M}
M:=Q-m_nE.
\end{equation}
Because $\langle E,X\rangle=1$ on the feasible set of \eqref{eq:dnnQ},
\[
\langle Q,X\rangle=m_n+\langle M,X\rangle,
\]
so exactness and uniqueness reduce to the shifted problem
\begin{equation}\label{eq:dnnM}
\vartheta(M)=\min\{\langle M,X\rangle:\ \langle E,X\rangle=1,\ X\succeq0,\ X\ge0\}.
\end{equation}

\begin{definition}[Defect graph]
The defect graph $G_n^-$ is the graph on $[n]$ with edge set
\[
E(G_n^-):=\bigl\{\{i,j\}:\ 1\le i<j\le n,\ M_{ij}<0\bigr\}=\bigl\{\{i,j\}:\ Q_{ij}<m_n\bigr\}.
\]
\end{definition}

We now state our main universal exactness result.

\begin{theorem}[Universal exactness theorem]\label{thm:universal}
Under Assumption~\ref{ass:generator},
\[
\Prob\bigl(\text{the DNN optimizer is unique and rank one}\bigr)\to 1.
\]
More precisely,
\[
\lim_{n\to\infty}\Prob\Bigl(\max_{C\in\mathcal{C}(G_n^-)}|C|\le 4\Bigr)=1,
\]
where $\mathcal{C}(G_n^-)$ denotes the set of connected components of the defect graph $G_n^-$. On that event, the global shifted DNN value equals $\min_{C\in\mathcal{C}(G_n^-)}\vartheta(M[C])$. Since $\DNN^r=\CP^r$ for $r\le4$, each local relaxation is exact, and absolute continuity implies that almost surely all local candidate values are distinct; hence the global DNN optimizer is unique and rank one.
\end{theorem}

The following theorem provides a quantitative bound for the GOE case, which is proved in Section~\ref{sec:goe}.

\begin{theorem}[GOE specialization]\label{thm:goe-main}
Assume $Q$ is GOE, i.e., the diagonal entries are i.i.d. $N(0,1)$ and the off-diagonal entries are i.i.d. $N(0,1/2)$, mutually independent. Then there exists $K<\infty$ such that
\[
\Prob\bigl(\text{the optimizer of \eqref{eq:dnnQ} is unique and rank one}\bigr)
\ge 1-K\frac{(\ln n)^2}{n^3}.
\]
Consequently, with the same probability, the DNN relaxation is exact and the original StQP has a unique optimizer.
\end{theorem}

\section{Deterministic decomposition over defect components}\label{sec:deterministic}
Let $C_1,\dots,C_t$ be the connected components of $G_n^-$. We use the notation introduced in Section~\ref{sec:model}; thus
\[
\langle E_{C_a},X[C_a]\rangle=\sum_{i,j\in C_a}X_{ij},
\qquad
\langle E_{C_a,C_b},X[C_a,C_b]\rangle=\sum_{i\in C_a}\sum_{j\in C_b}X_{ij}.
\]

\begin{proposition}[Global value equals the best local value]\label{prop:global-local}
We have
\[
\vartheta(M)=\min_{1\le a\le t}\vartheta(M[C_a]).
\]
Moreover, if $Y^\star$ is optimal for a component $C_a$ attaining the minimum local value, then its block embedding into $[n]\times[n]$ is globally feasible and optimal for \eqref{eq:dnnM}.
\end{proposition}

\begin{proof}
Set
\[
\mu:=\min_{1\le a\le t}\vartheta(M[C_a]).
\]
Embedding any local optimizer for a component attaining the minimum local value yields a globally feasible matrix with objective value $\mu$, so $\vartheta(M)\le \mu$.

Conversely, let $X$ be any feasible matrix for \eqref{eq:dnnM}. For each component $C_a$, define the block mass
\[
\tau_a:=\langle E_{C_a},X[C_a]\rangle\ge0.
\]
Let
\[
\eta:=\sum_{a\neq b}\langle E_{C_a,C_b},X[C_a,C_b]\rangle\ge0.
\]
Since $X\ge0$ and $\langle E,X\rangle=1$, we have
\[
1=\sum_{a=1}^t\tau_a+\eta.
\]
If $\tau_a>0$, then $Y^{(a)}:=X[C_a]/\tau_a$ is feasible for the local problem on $C_a$, so
\[
\langle M[C_a],X[C_a]\rangle\ge \tau_a\vartheta(M[C_a])\ge \tau_a\mu.
\]
Summing gives
\begin{equation}\label{eq:block-lb-univ}
\sum_{a=1}^t \langle M[C_a],X[C_a]\rangle\ge (1-\eta)\mu.
\end{equation}
Now if $i$ and $j$ lie in different connected components of $G_n^-$, then by definition $M_{ij}\ge0$. Since also $X_{ij}\ge0$, the cross-block contribution to the objective is nonnegative, and therefore
\begin{equation}\label{eq:cross-nonneg-univ}
\langle M,X\rangle\ge \sum_{a=1}^t \langle M[C_a],X[C_a]\rangle.
\end{equation}
Combining \eqref{eq:block-lb-univ} and \eqref{eq:cross-nonneg-univ} gives
\[
\langle M,X\rangle\ge (1-\eta)\mu.
\]
Finally, $\mu\le0$ because some index $i_\star$ satisfies $M_{i_\star i_\star}=0$ by construction. If $i_\star\in C_a$, then the local feasible matrix $e_{i_\star}e_{i_\star}^\top$ on block $C_a$ attains value $0$, so $\vartheta(M[C_a])\le0$ and therefore $\mu\le0$. Since $\eta\ge0$ and $\mu\le0$, we have $(1-\eta)\mu\ge\mu$. Hence $\langle M,X\rangle\ge\mu$ for every feasible $X$, which proves $\vartheta(M)\ge\mu$.

Together with the first paragraph, this gives the claimed identity and the optimality of the embedded local optimizer.
\end{proof}

\begin{corollary}[Equality characterization]\label{cor:equality}
Let $\mu=\min_a \vartheta(M[C_a])$ and let $X$ be feasible for \eqref{eq:dnnM} with $\langle M,X\rangle=\mu$. Define $\tau_a$ and $\eta$ as above. Then:
\begin{enumerate}[label=(\roman*)]
\item if $\tau_a>0$, then $\vartheta(M[C_a])=\mu$ and $X[C_a]/\tau_a$ is locally optimal on $C_a$;
\item if $\mu<0$, then necessarily $\eta=0$;
\item if, in addition, every cross-component entry of $M$ is strictly positive and there is a unique component attaining the minimum local value, then $X$ is supported entirely on that component.
\end{enumerate}
\end{corollary}

\begin{proof}
If some active block $C_a$ had value strictly larger than $\mu$, the proof of Proposition~\ref{prop:global-local} would dictate
\[
\langle M,X\rangle > (1-\eta)\mu\ge \mu,
\]
which is a contradiction. This proves (i). If $\mu<0$ and $\eta>0$, then $(1-\eta)\mu>\mu$, again impossible; this is (ii). For (iii), since $\mu\le 0$ and $\eta \ge 0$, we have $(1-\eta)\mu \ge \mu$. Combining this with the block lower bounds \eqref{eq:block-lb-univ} and the cross-component decomposition of $\langle M,X\rangle$ yields
\[
\begin{aligned}
\langle M,X\rangle
&= \sum_{a=1}^t \langle M[C_a],X[C_a]\rangle
 + \sum_{a\neq b}\langle M[C_a,C_b],X[C_a,C_b]\rangle\\
&\ge (1-\eta)\mu
 + \sum_{a\neq b}\langle M[C_a,C_b],X[C_a,C_b]\rangle\\
&\ge \mu
 + \sum_{a\neq b}\langle M[C_a,C_b],X[C_a,C_b]\rangle.
\end{aligned}
\]
If $\langle M,X\rangle = \mu$, the cross-component sum must be nonpositive. Since $M_{ij} > 0$ strictly for cross-component entries and $X_{ij} \ge 0$, all cross-component blocks $X[C_a,C_b]$ must be identically zero, which forces $\eta=0$. By (i), every active block must attain the minimal local value. Since exactly one component attains this value, denote it by $C_\star$; only $C_\star$ can be active. For any inactive block $C_a$ with $a\neq \star$, its block mass satisfies $\tau_a=\langle E_{C_a},X[C_a]\rangle=0$. Since $X\ge0$ entrywise, every entry of $X[C_a]$ is zero. Therefore every entry of $X$ outside the $C_\star\times C_\star$ block is zero, so $X$ is supported entirely on $C_\star$.
\end{proof}

\section{Local exactness and local uniqueness on blocks of size at most four}\label{sec:smallblocks}
We rely on the classical low-dimensional equivalence between the doubly nonnegative and completely positive cones.

\begin{theorem}[Diananda~\cite{Diananda}; Berman and Shaked-Monderer~\cite{BermanShaked}]\label{thm:cpdnn}
For every $r\le4$,
\[
\DNN^r=\CP^r.
\]
Equivalently, every $r\times r$ matrix that is positive semidefinite and entrywise nonnegative can be decomposed, for some positive integer $N$, as
\[
X=\sum_{s=1}^N y_sy_s^\top,\qquad y_s\in\R_+^r.
\]
\end{theorem}

Fix a nonempty subset $C\subseteq[n]$ with $|C|\le4$. Write $r:=|C|$ and $A:=Q[C]$.

\subsection{Local exactness}
For the local matrix $A$, the quantities $\mu(A)$ and $\vartheta(A)$ are the StQP and DNN values defined in Section~\ref{sec:model}.

\begin{lemma}[Local DNN exactness on small blocks]\label{lem:local-exact}
If $|C|\le4$, then
\[
\vartheta(A)=\mu(A),
\]
and the local DNN problem admits at least one rank-one optimizer $xx^\top$ with $x\in\Delta^{r-1}$ minimizing the local StQP.
\end{lemma}

\begin{proof}
Every feasible matrix $Y$ for the local DNN problem lies in $\CP^r$ by Theorem~\ref{thm:cpdnn}, so for some positive integer $N$,
\[
Y=\sum_{s=1}^N y_sy_s^\top,\qquad y_s\in\R_+^r.
\]
By omitting any zero vectors from the decomposition, we may assume $y_s\neq 0$ for all $s$. Let $\alpha_s:=\1^\top y_s>0$ and set $x_s:=y_s/\alpha_s\in\Delta^{r-1}$. Since
\[
\langle E,Y\rangle=\sum_{s=1}^N \alpha_s^2,
\]
feasibility implies $\sum_{s=1}^N \alpha_s^2=1$. Therefore
\[
\langle A,Y\rangle=\sum_{s=1}^N \alpha_s^2 x_s^\top A x_s\ge \mu(A).
\]
Conversely, if $x^\star$ minimizes the local StQP, then $x^\star(x^\star)^\top$ is feasible for the local DNN problem and attains value $\mu(A)$.
\end{proof}

\subsection{Finite support-candidate analysis}
For every nonempty support set $S\subseteq C$, let
\[
\Delta_S:=\{x\in\Delta^{r-1}:\ \supp(x)\subseteq S\},
\]
and let
\[
\mu_S(A):=\min\{x^\top A x:\ x\in\Delta_S\}.
\]
Then
\[
\mu(A)=\min_{\varnothing\neq S\subseteq C}\mu_S(A).
\]

For $|S|=s\ge2$, write $A_S:=A[S]$ and choose once and for all a deterministic matrix $B_S\in\R^{s\times(s-1)}$ whose columns form an orthonormal basis of the tangent space
\[
T_S:=\{u\in\R^s:\ \1^\top u=0\}.
\]
Define the KKT matrix
\[
K_S(A):=
\begin{pmatrix}
A_S & -\1\\
\1^\top & 0
\end{pmatrix}.
\]
On the event $\det K_S(A)\neq0$, let $(x_S(A),\lambda_S(A))$ denote the unique solution of
\[
A_Sx=\lambda\1,\qquad \1^\top x=1.
\]
Here $x_S(A)$ is a vector in the active coordinate space $\R^{|S|}$; when it is used as a vector in $\R^r$, it is embedded by placing its entries on $S$ and zeros on $C\setminus S$. We call $S$ \emph{admissible} if either $|S|=1$, or $|S|\ge2$ and
\[
\det K_S(A)\neq0,\qquad x_S(A)\in\R_{++}^{|S|},\qquad B_S^\top A_SB_S\succ0.
\]
For a singleton $S=\{i\}$, define $x_S(A)=(1)\in\R$ and $\lambda_S(A)=A_{ii}$.

\begin{lemma}[Face values are represented by admissible support candidates]\label{lem:face-cand}
Almost surely, for every nonempty $S\subseteq C$,
\[
\mu_S(A)=\min\{\lambda_T(A):\ \varnothing\neq T\subseteq S,\ T\text{ admissible}\}.
\]
Moreover, whenever $T\subseteq S$ is admissible, the embedded vector $x_T(A)$ is the unique minimizer on the face $\Delta_T$.
\end{lemma}

\begin{proof}
Let $\Omega_0$ be the event that for every nonempty $T\subseteq C$ with $|T|\ge2$,
\[
\det(B_T^\top A_T B_T)\neq0.
\]
For each fixed $T$, the determinant is a polynomial in the entries of $A$ and is not identically zero because for $A=I$ one has $B_T^\top A_T B_T=I_{|T|-1}$. Since there are finitely many $T$, we have $\Prob(\Omega_0)=1$.

Fix $S\subseteq C$ and work on $\Omega_0$. Let $x^\star\in\Delta_S$ minimize $x^\top A x$ over $\Delta_S$, and let $T:=\supp(x^\star)\subseteq S$. Then $x^\star$ lies in the relative interior of $\Delta_T$ and minimizes the quadratic form on that face. Identifying $x^\star$ with its restriction to the active coordinates $T$, the KKT conditions on $\Delta_T$ give
\[
A_Tx^\star=\lambda\1,\qquad \1^\top x^\star=1
\]
for some $\lambda\in\R$. The second-order necessary condition on the face yields
\[
u^\top A_Tu\ge0\qquad (u\in T_T).
\]
The quadratic form is positive semidefinite on $T_T$, and since $\Omega_0$ excludes singularity of $B_T^\top A_T B_T$, it is strictly positive definite on $T_T$. It follows that $K_T(A)$ is invertible: if $K_T(A)\begin{pmatrix}u\\ \eta\end{pmatrix}=0$, then $u\in T_T$ and
\[
u^\top A_Tu = \eta\,\1^\top u =0,
\]
so $u=0$ and then $\eta=0$. Thus $T$ is admissible and $(x^\star,\lambda)=(x_T(A),\lambda_T(A))$.

Conversely, if $T$ is admissible and $z=x_T(A)+u$ with $u\in T_T$, then
\[
z^\top A_T z = x_T(A)^\top A_T x_T(A) + u^\top A_Tu = \lambda_T(A)+u^\top A_Tu\ge \lambda_T(A),
\]
with equality only when $u=0$. Therefore $x_T(A)$ is the unique minimizer on $\Delta_T$. Since every minimizer on $\Delta_S$ arises from some support $T\subseteq S$, the claimed representation follows.
\end{proof}

\begin{lemma}[Distinct admissible candidates do not tie]\label{lem:cand-notie}
Fix distinct nonempty sets $S,T\subseteq C$. Then
\[
\Prob\bigl(\lambda_S(A)=\lambda_T(A),\ S\text{ admissible},\ T\text{ admissible}\bigr)=0.
\]
\end{lemma}

\begin{proof}
If both supports are singletons, the equality event is a diagonal tie and has probability zero by continuity. Assume at least one support has size at least two. On the open set where the relevant KKT matrices are invertible, $\lambda_S$ and $\lambda_T$ are rational functions of the finitely many entries of $A$. Multiplying by the nonzero determinant denominators shows that the equality event is contained in the zero set of a polynomial. It is enough to prove that this polynomial is not identically zero.

Take the diagonal matrix
\[
D_C:=\diag(2,4,\dots,2^{|C|})
\]
on the ordered coordinates of $C$. For every nonempty $U\subseteq C$, the KKT system on $U$ is invertible and
\[
\lambda_U(D_C)=\bigl(\1^\top D_C[U]^{-1}\1\bigr)^{-1}
=\Bigl(\sum_{i\in U}2^{-p_C(i)}\Bigr)^{-1},
\]
where $p_C(i)\in\{1,\dots,|C|\}$ is the position of $i$ inside the ordered set $C$. $D_C[U]^{-1}$ is diagonal with entries $2^{-p_C(i)}$, so $\1^\top D_C[U]^{-1}\1$ is exactly the sum of those diagonal entries. These numbers are distinct for distinct $U$ because distinct subsets of dyadic fractions have distinct sums. For $|U|\ge2$, also $x_U(D_C)\in\R_{++}^{|U|}$ and $B_U^\top D_C[U]B_U\succ0$; singleton supports are admissible by definition. Thus every $U$ is admissible at $D_C$. Therefore $\lambda_S(D_C)\neq\lambda_T(D_C)$, and the resulting polynomial is nonzero. Since the law of $A$ is absolutely continuous, the zero set has probability zero.
\end{proof}

\begin{proposition}[Almost-sure uniqueness of the local StQP]\label{prop:local-stqp-unique}
For every fixed $C\subseteq[n]$ with $|C|\le4$, the local StQP
\[
\min\{x^\top Q[C]x:\ x\in\Delta^{|C|-1}\}
\]
has a unique minimizer almost surely.
\end{proposition}

\begin{proof}
By Lemma~\ref{lem:face-cand} with $S=C$, almost surely the block value is the minimum of finitely many admissible candidate values $\lambda_T(A)$. By Lemma~\ref{lem:cand-notie}, two distinct admissible candidates do not tie almost surely. Hence almost surely there is a unique admissible support $T_\star$ that attains the minimum. Lemma~\ref{lem:face-cand} then shows that the embedded vector $x_{T_\star}(A)$ is the unique minimizer.
\end{proof}

A real-valued random variable $Z$ is called \emph{atomless} if $\Prob(Z=t)=0$ for every $t\in\R$.

\begin{lemma}[The local StQP value is atomless]\label{lem:atomless}
For every fixed $C\subseteq[n]$ with $|C|\le4$, the random variable $\mu(Q[C])$ is atomless.
\end{lemma}

\begin{proof}
Fix $t\in\R$. By Proposition~\ref{prop:local-stqp-unique}, almost surely the unique minimizing support is some admissible $T\subseteq C$. Therefore
\[
\{\mu(Q[C])=t\}
\subseteq
\bigcup_{\varnothing\neq T\subseteq C}\{\lambda_T(Q[C])=t,\ T\text{ admissible}\}
\cup \mathcal N,
\]
where $\mathcal N$ is the finite union of the zero-probability exceptional events excluded in Lemma~\ref{lem:face-cand} and Proposition~\ref{prop:local-stqp-unique}; hence $\Prob(\mathcal N)=0$.

For a singleton $T$, $\lambda_T(Q[C])$ is absolutely continuous and therefore atomless. For $|T|\ge2$, on the open set where $K_T$ is invertible, multiplying the identity $\lambda_T=t$ by the nonzero determinant denominator gives a polynomial equation in the entries of $Q[C]$. This polynomial is not identically zero because $\lambda_T$ is nonconstant: if $Q[C]=cI$ with $c>0$, then $T$ is admissible and the KKT system yields the equally weighted vector on $T$ (that is, the vector with all active coordinates equal to $1/|T|$), hence $\lambda_T(Q[C])=c/|T|$. Choosing $c=|T|(|t|+1)$ gives a point where $\lambda_T\neq t$. Thus each level set has probability zero, and the union is finite.
\end{proof}

\begin{corollary}[Almost-sure uniqueness of the local DNN optimizer]\label{cor:local-dnn-unique}
For every fixed $C\subseteq[n]$ with $|C|\le4$, the local DNN problem on $Q[C]$ has a unique optimizer almost surely, and that optimizer is rank one.
\end{corollary}

\begin{proof}
By Proposition~\ref{prop:local-stqp-unique}, the local StQP has a unique minimizer $x_C$ almost surely. By Lemma~\ref{lem:local-exact}, the rank-one matrix $x_Cx_C^\top$ is locally DNN-optimal. If $Y$ is any other local DNN optimizer, write it as a completely positive decomposition, for some positive integer $N$,
\[
Y=\sum_{s=1}^N y_sy_s^\top
\]
using Theorem~\ref{thm:cpdnn}. The proof of Lemma~\ref{lem:local-exact} shows that every simplex vector arising from this decomposition must itself minimize the local StQP. By uniqueness, every such vector equals $x_C$, and therefore $Y=x_Cx_C^\top$.
\end{proof}

\begin{lemma}[No ties between different small components]\label{lem:notie}
Fix disjoint nonempty sets $C_1,C_2\subseteq[n]$ with $|C_1|,|C_2|\le4$. Then
\[
\Prob\bigl(\vartheta(Q[C_1])=\vartheta(Q[C_2])\bigr)=0.
\]
Consequently, almost surely no two disjoint subsets of size at most four have the same local DNN value.
\end{lemma}

\begin{proof}
By Corollary~\ref{cor:local-dnn-unique}, almost surely
\[
\vartheta(Q[C_1])=\mu(Q[C_1]),\qquad \vartheta(Q[C_2])=\mu(Q[C_2]).
\]
The random variables $\mu(Q[C_1])$ and $\mu(Q[C_2])$ depend on disjoint entry sets of $Q$, hence are independent. By Lemma~\ref{lem:atomless}, each is atomless. Therefore
\[
\Prob\bigl(\mu(Q[C_1])=\mu(Q[C_2])\mid \mu(Q[C_1])\bigr)=0
\]
almost surely, and taking expectations proves the claim. The final statement follows from a finite union bound over all disjoint pairs $(C_1,C_2)$ with $|C_1|,|C_2|\le4$.
\end{proof}

\section{Proof of the universal theorem}\label{sec:universal}
To prove the universal theorem, we first bound the probability that the defect graph contains large connected components.

\begin{lemma}[Five-tree bound in the universal model]\label{lem:five-tree-univ}
Under Assumption~\ref{ass:generator},
\[
\Prob\bigl(\exists\text{ a connected component of }G_n^-\text{ of size }\ge5\bigr)\to 0.
\]
More precisely,
\[
\Prob\bigl(\exists\text{ a connected component of }G_n^-\text{ of size }\ge5\bigr)
\le \binom{n}{5}5^{3}\,\E\bigl[F_O(m_n)^4\bigr].
\]
\end{lemma}

\begin{proof}
Conditioned on $m_n$, the off-diagonal entries remain independent and
\[
\Prob(Q_{ij}<m_n\mid m_n)=F_O(m_n).
\]
Thus $G_n^-\mid m_n$ is an Erd\H{o}s-R\'enyi graph $G\bigl(n,F_O(m_n)\bigr)$.

Any connected graph on at least five vertices contains a tree on exactly five vertices: start from a spanning tree of the component and repeatedly delete leaves until only five vertices remain. Therefore the event of a component of size at least five is contained in the event that $G_n^-$ contains some labeled tree on five vertices.

Fix a five-element set $S\subseteq[n]$. A labeled tree here means an undirected tree on the fixed labeled vertex set $S$. By Cayley's formula \cite{Cayley1889}, the number of labeled trees on $S$ is $5^{5-2}=125$. Conditioned on $m_n$, any prescribed five-vertex tree appears with probability $F_O(m_n)^4$. Therefore
\[
\Prob\bigl(\exists\text{ a five-vertex tree in }G_n^-\mid m_n\bigr)
\le \binom{n}{5}125\,F_O(m_n)^4.
\]
Taking expectations proves the bound, and Assumption~\ref{ass:generator}(A3) yields convergence to zero.
\end{proof}

\begin{proof}[Proof of Theorem~\ref{thm:universal}]
Let $\mathcal E_n$ denote the event that every connected component of $G_n^-$ has size at most $4$. By Lemma~\ref{lem:five-tree-univ}, $\Prob(\mathcal E_n)\to1$.

Work on $\mathcal E_n$. Since the off-diagonal law is absolutely continuous and independent of the diagonals, almost surely $Q_{ij}\neq m_n$ for all $i<j$. Hence every cross-component entry of $M$ is strictly positive.

Each connected component $C_a$ has size at most $4$. By Corollary~\ref{cor:local-dnn-unique} and a finite union bound over all subsets of size at most four, almost surely every such subset has a unique rank-one local DNN optimizer. Since
\[
\vartheta(M[C_a])=\vartheta(Q[C_a])-m_n,
\]
Lemma~\ref{lem:notie} implies that almost surely the local shifted DNN values are pairwise distinct across different components. Thus there is almost surely a unique component attaining the minimum local value, say $C_\star$.

By Proposition~\ref{prop:global-local}, the global shifted DNN value equals the minimum local shifted value. Corollary~\ref{cor:equality}(iii) then shows that every global optimizer is supported entirely on $C_\star$. Since the local DNN optimizer on $C_\star$ is unique and rank one, the global shifted DNN optimizer is unique and rank one as well. Because shifting by $m_nE$ changes the objective by a constant only, the same statement holds for the original DNN problem \eqref{eq:dnnQ}.
\end{proof}

\section{GOE specialization and quantitative bound}\label{sec:goe}
In this section, we specialize to the GOE: the diagonal entries are i.i.d. $N(0,1)$ and the off-diagonal entries are i.i.d. $N(0,1/2)$, mutually independent.

Conditioned on the diagonal minimum $m_n$, the defect graph is an Erd\H{o}s-R\'enyi graph $G(n,q_n)$ with
\[
q_n:=\Phi(\sqrt2\,m_n),
\]
where $\Phi$ denotes the standard normal distribution function.

The following lemma establishes the explicit fourth-moment bound required for Theorem~\ref{thm:goe-main}.

\begin{lemma}[Moments of the GOE defect-edge probability]\label{lem:moments-goe}
For every fixed integer $s\ge1$ there exists $K_s<\infty$ such that
\[
\E[q_n^s]\le K_s\frac{(\ln n)^{s/2}}{n^{2s}}.
\]
In particular, taking $s=4$,
\[
\E[q_n^4]\le K_4\frac{(\ln n)^2}{n^8}.
\]
\end{lemma}

\begin{proof}
Let $U_i:=\Phi(Q_{ii})$. Then the $U_i$ are i.i.d. uniform on $(0,1)$ and $U_{(1)}:=\min_{1\le i\le n} U_i=\Phi(m_n)$. Equivalently, $U_{(1)}$ has density $n(1-u)^{n-1}$ on $(0,1)$. Moreover,
\[
q_n=\Phi\bigl(\sqrt2\,\Phi^{-1}(U_{(1)})\bigr).
\]
We use the standard Mills-ratio inequalities (see Mills~\cite{Mills1926} and Gordon~\cite{Gordon1941})
\[
\frac{x}{1+x^2}\,\varphi(x)<1-\Phi(x)<\frac{\varphi(x)}{x},\qquad x>0,
\]
where $\varphi$ denotes the standard normal density. Let $u\in(0,1/2)$ and set $x:=-\Phi^{-1}(u)>0$, so that $u=\Phi(-x)$. Then the inequalities above give
\[
u\asymp \frac{\varphi(x)}{x}.
\]
Applying the upper bound again at $\sqrt2\,x$ yields
\[
\Phi(-\sqrt2\,x)\le \frac{\varphi(\sqrt2\,x)}{\sqrt2\,x}
=\frac{1}{2\sqrt\pi\,x}e^{-x^2}.
\]
On the other hand, $u\asymp \varphi(x)/x=(2\pi)^{-1/2}x^{-1}e^{-x^2/2}$, hence
\[
u^2\asymp \frac{1}{2\pi x^2}e^{-x^2}.
\]
Combining the last two displays gives
\[
\Phi(-\sqrt2\,x)\le K_0 x u^2.
\]
Finally, the relation $u\asymp \varphi(x)/x$ implies $x\asymp \sqrt{\ln(1/u)}$ as $u\downarrow 0$, so for all sufficiently small $u>0$,
\[
\Phi\bigl(\sqrt2\,\Phi^{-1}(u)\bigr)=\Phi(-\sqrt2\,x)\le K_0 u^2\sqrt{\ln(1/u)}.
\]
Hence there exists $u_0\in(0,1/2]$ and $K_0<\infty$ such that the displayed bound holds on $(0,u_0]$; enlarging $K_0$ if necessary, the same bound holds on $(0,1/2]$.

Split the expectation at $u=1/2$. On $(0,1/2]$ we obtain
\[
q_n^s\le K_s^{(1)} U_{(1)}^{2s}(\ln(1/U_{(1)}))^{s/2},
\]
where $K_s^{(1)}:=K_0^s$. Enlarging $K_s^{(1)}$ if necessary so that $K_s^{(1)}\ge1$, and using the density of $U_{(1)}$ displayed above, we have
\[
\E[q_n^s]\le K_s^{(1)} n\int_0^{1/2} u^{2s}(\ln(1/u))^{s/2}(1-u)^{n-1}\,du
+ K_s^{(1)} n\int_{1/2}^1 (1-u)^{n-1}\,du.
\]
The second integral contribution is $K_s^{(1)}(1/2)^n$. For the first integral, using $(1-u)^{n-1}\le e^{-(n-1)u}$ and changing variables $v=(n-1)u$ yields a constant $K_s^{(2)}<\infty$, depending only on $s$, such that
\[
\E[q_n^s]\le K_s^{(2)}\left[
 n^{-2s}\int_0^{\infty} v^{2s}\bigl(\ln((n-1)/v)\bigr)_+^{s/2}e^{-v}\,dv
+ (1/2)^n\right].
\]
The integral is $O_s((\ln n)^{s/2})$, while the exponential term is negligible relative to $n^{-2s}(\ln n)^{s/2}$. Therefore, for some $K_s^{(3)}<\infty$,
\[
\E[q_n^s]\le K_s^{(3)}\frac{(\ln n)^{s/2}}{n^{2s}}.
\]
Renaming $K_s^{(3)}$ as $K_s$ proves the claim.
\end{proof}

\begin{theorem}[Quantitative five-tree estimate under GOE]\label{thm:goe-fivetree}
There exists $K<\infty$ such that
\[
\Prob\bigl(\exists\text{ a connected component of }G_n^-\text{ of size }\ge5\bigr)
\le K\frac{(\ln n)^2}{n^3}.
\]
\end{theorem}

\begin{proof}
By Lemma~\ref{lem:five-tree-univ} and Lemma~\ref{lem:moments-goe} with $s=4$,
\[
\begin{aligned}
&\Prob\bigl(\exists\text{ a connected component of }G_n^-\text{ of size }\ge5\bigr)\\
&\quad\le \binom{n}{5}125\,\E[q_n^4]
\le K n^5\cdot \frac{(\ln n)^2}{n^8}
= K\frac{(\ln n)^2}{n^3}.
\end{aligned}
\]
\end{proof}

\begin{proof}[Proof of Theorem~\ref{thm:goe-main}]
Theorem~\ref{thm:goe-fivetree} shows that the event in Theorem~\ref{thm:universal} has probability at least $1-K(\ln n)^2/n^3$. On that event, Theorem~\ref{thm:universal} yields uniqueness and rank-one exactness of the DNN optimizer. The final statement about the original StQP follows exactly as in the last paragraph of the proof of Theorem~\ref{thm:universal}.
\end{proof}

\section{Extensions and open directions}\label{sec:extensions}
The universal theorem isolates the single probabilistic quantity
\[
 n^5\,\E\bigl[F_O(m_n)^4\bigr]
\]
that controls the absence of five-vertex defect trees. We now verify the tail-decay condition (A3) for several broader model classes. A positive function $h$ is regularly varying at infinity with index $\rho$ (see, e.g., de Haan and Ferreira~\cite{deHaanFerreira2006}, Leadbetter et al.~\cite{LeadbetterLindgrenRootzen1983}, or Resnick~\cite{Resnick1987}) if $h(tx)/h(x)\to t^\rho$ for every $t>0$ as $x\to\infty$; at a finite endpoint, regular variation with index $\rho$ means $h(tx)/h(x)\to t^\rho$ for every $t>0$ as $x\downarrow0$. In Propositions~\ref{prop:heavytail} and~\ref{prop:uniformtype}, the hypotheses are the concrete pure-power asymptotics displayed there. For the quantile-based arguments below, $F^{\leftarrow}(u):=\inf\{x:\ F(x)\ge u\}$ denotes the lower quantile function. Theorem~\ref{thm:universal} then implies exactness and uniqueness with probability $1-o(1)$.

\begin{proposition}[Variance-tuned Gaussian Wigner models]\label{prop:vtwigner}
Assume the diagonal entries are i.i.d. $N(0,\gamma^2)$ and the off-diagonal entries are i.i.d. $N(0,\sigma^2)$, mutually independent. If
\[
\sigma^2< \frac45\,\gamma^2,
\]
then Assumption~\ref{ass:generator}(A3) holds.
\end{proposition}

\begin{proof}
Set $a:=\gamma/\sigma$. Since the diagonal law is $N(0,\gamma^2)$, the variables $U_i:=\Phi(Q_{ii}/\gamma)$ are i.i.d. uniform on $(0,1)$. Thus
\[
U_{(1)}:=\min_{1\le i\le n} U_i=\Phi(m_n/\gamma),
\]
and $U_{(1)}$ has the same order-statistic distribution as in Lemma~\ref{lem:moments-goe}. Moreover,
\[
F_O(m_n)=\Phi(m_n/\sigma)=\Phi\bigl(a\,\Phi^{-1}(U_{(1)})\bigr).
\]
Let $x=-\Phi^{-1}(u)$; as $u\downarrow0$, $x\to\infty$ and $\Phi\bigl(a\Phi^{-1}(u)\bigr)=\Phi(-a x)$. By Mills' ratio, for all sufficiently large $x$,
\[
u=\Phi(-x)\asymp x^{-1}e^{-x^2/2},\qquad
\Phi(-a x)\le K_a^{(1)} x^{-1}e^{-a^2x^2/2}.
\]
The first relation implies $e^{-x^2/2}\asymp x u$, and therefore
\[
e^{-a^2x^2/2}\le K_a^{(2)} x^{a^2}u^{a^2}.
\]
Substitution into the second inequality yields
\[
\Phi(-a x)\le K_a^{(3)} u^{a^2}x^{a^2-1}.
\]
Since $x\asymp\sqrt{\ln(1/u)}$ as $u\downarrow0$, we obtain
\[
\Phi\bigl(a\Phi^{-1}(u)\bigr)
\le K_a^{(4)} u^{a^2}\{\ln(1/u)\}^{(a^2-1)/2}
\]
for all sufficiently small $u$. Hence,
\[
\Phi\bigl(a\Phi^{-1}(u)\bigr)^4
\le K_a^{(5)} u^{4a^2}\{\ln(1/u)\}^{2a^2-2}.
\]
Since $F_O(m_n)=\Phi(a\Phi^{-1}(U_{(1)}))$ and $U_{(1)}$ has density $n(1-u)^{n-1}$, the expectation can be written as
\[
\E[F_O(m_n)^4]
= n\int_0^1 \Phi(a\Phi^{-1}(u))^4(1-u)^{n-1}\,du.
\]
Splitting this integral at a fixed small $u_0\in(0,1/2]$ and using the preceding bound on $(0,u_0]$, while bounding the contribution from $[u_0,1]$ by an exponentially small term, gives
\[
\E[F_O(m_n)^4]
\le K_a^{(6)} n\int_0^{u_0}u^{4a^2}\{\ln(1/u)\}^{2a^2-2}(1-u)^{n-1}\,du+O(e^{-kn}).
\]
Under the hypothesis $a^2>5/4$, set $\beta:=2a^2-2>0$. Using $(1-u)^{n-1}\le e^{-(n-1)u}$ and the change of variables $v=(n-1)u$, the integral term satisfies
\[
\begin{aligned}
&n\int_0^{u_0}u^{4a^2}\{\ln(1/u)\}^{\beta}(1-u)^{n-1}\,du\\
&\quad\le n(n-1)^{-(4a^2+1)}
\int_0^{(n-1)u_0}v^{4a^2}\bigl\{\ln\bigl((n-1)/v\bigr)\bigr\}^{\beta}e^{-v}\,dv.
\end{aligned}
\]
On this domain, $v\le (n-1)u_0$ and $u_0\le 1/2$, so $(n-1)/v\ge 1/u_0\ge 2$. Thus the logarithm is strictly positive, and we obtain the bound
\[
\bigl\{\ln\bigl((n-1)/v\bigr)\bigr\}^{\beta}
\le K_a^{(7)}(\ln n)^{\beta}(1+|\ln v|^{\beta}).
\]
Hence
\[
\begin{aligned}
&n\int_0^{u_0}u^{4a^2}\{\ln(1/u)\}^{\beta}(1-u)^{n-1}\,du\\
&\quad\le K_a^{(8)} n^{-4a^2}(\ln n)^{\beta}
\int_0^\infty v^{4a^2}(1+|\ln v|^{\beta})e^{-v}\,dv.
\end{aligned}
\]
The remaining integral is finite, so
\[
 n^5\E[F_O(m_n)^4]
\le K_a^{(9)} n^{5-4a^2}(\ln n)^{2a^2-2}.
\]
If $\sigma^2<4\gamma^2/5$, then $a^2>5/4$, so the polynomial factor $n^{5-4a^2}$ tends to zero fast enough to dominate the logarithmic factor. Thus Assumption~\ref{ass:generator}(A3) holds.
\end{proof}

\begin{proposition}[Heavy-tailed ensembles]\label{prop:heavytail}
Assume the diagonal and off-diagonal laws have left tails satisfying
\[
F_D(-x)\sim c_Dx^{-\alpha_D},\qquad F_O(-x)\sim c_Ox^{-\alpha_O},\qquad x\to\infty,
\]
with $c_D,c_O>0$ and $\alpha_D,\alpha_O>0$. If
\[
\alpha_O>\frac54\,\alpha_D,
\]
then Assumption~\ref{ass:generator}(A3) holds.
\end{proposition}

\begin{proof}
By continuity of $F_D$, the variables $U_i:=F_D(Q_{ii})$ are i.i.d. uniform on $(0,1)$, and $U_{(1)}:=\min_{1\le i\le n} U_i=F_D(m_n)$ has density $n(1-u)^{n-1}$ on $(0,1)$. Equivalently, the minimum has the same distribution as $F_D^{\leftarrow}(U_{(1)})$, and hence
\[
\E[F_O(m_n)^4]
= n\int_0^1 F_O(F_D^{\leftarrow}(u))^4(1-u)^{n-1}\,du.
\]
The asymptotic equivalence $F_D(-x)\sim c_Dx^{-\alpha_D}$ implies, by standard quantile inversion,
\[
-F_D^{\leftarrow}(u)\sim c_D^{1/\alpha_D}u^{-1/\alpha_D}\qquad (u\downarrow0).
\]
Evaluating the off-diagonal tail at this lower quantile gives
\[
F_O\bigl(F_D^{\leftarrow}(u)\bigr)
\sim c_Oc_D^{-\alpha_O/\alpha_D}u^{\alpha_O/\alpha_D}\qquad (u\downarrow0).
\]
Thus, by the definition of asymptotic equivalence, there exist $u_0\in(0,1)$ and $K<\infty$ such that
\[
F_O\bigl(F_D^{\leftarrow}(u)\bigr)^4\le K u^{4\alpha_O/\alpha_D}\qquad (0<u\le u_0).
\]
Let $p:=4\alpha_O/\alpha_D$. In the expectation integral displayed above, the contribution from $u\in(u_0,1)$ is at most exponentially small, because $F_O\le 1$ and
\[
\int_{u_0}^1 n(1-u)^{n-1}\,du=(1-u_0)^n.
\]
For the integral over $(0,u_0]$, using $(1-u)^{n-1}\le e^{-(n-1)u}$ gives
\[
\begin{aligned}
 n\int_0^{u_0}u^p(1-u)^{n-1}\,du
&\le n\int_0^\infty u^p e^{-(n-1)u}\,du\\
&=n(n-1)^{-(p+1)}\Gamma(p+1)\\
&\le K_p n^{-p}.
\end{aligned}
\]
Consequently, $\E[F_O(m_n)^4] \le K_p n^{-p}+O(e^{-kn})=O(n^{-p})$. Since $\alpha_O>5\alpha_D/4$, we have $p>5$, and therefore $n^5\E[F_O(m_n)^4]=O(n^{5-p})\to0$.
\end{proof}

\begin{proposition}[Finite-lower-endpoint and uniform-type laws]\label{prop:uniformtype}
Assume the diagonal and off-diagonal laws have a common finite lower endpoint $-a$ and satisfy
\[
F_D(-a+t)\sim c_D t^{\beta_D},\qquad F_O(-a+t)\sim c_O t^{\beta_O},\qquad t\downarrow0,
\]
with $c_D,c_O>0$ and $\beta_D,\beta_O>0$. If
\[
\beta_O>\frac54\,\beta_D,
\]
then Assumption~\ref{ass:generator}(A3) holds.
\end{proposition}

\begin{proof}
Again let $U_{(1)}:=F_D(m_n)$, so $U_{(1)}$ has density $n(1-u)^{n-1}$, $m_n$ has the same distribution as $F_D^{\leftarrow}(U_{(1)})$, and
\[
\E[F_O(m_n)^4]
= n\int_0^1 F_O(F_D^{\leftarrow}(u))^4(1-u)^{n-1}\,du.
\]
The endpoint relation implies
\[
F_D^{\leftarrow}(u)+a\sim c_D^{-1/\beta_D}u^{1/\beta_D}\qquad (u\downarrow0).
\]
Evaluating the $F_O$ asymptotic at this quantile yields
\[
F_O\bigl(F_D^{\leftarrow}(u)\bigr)
\sim c_Oc_D^{-\beta_O/\beta_D}u^{\beta_O/\beta_D}\qquad (u\downarrow0).
\]
Setting $p:=4\beta_O/\beta_D$, the same asymptotic power bound for $F_O(F_D^{\leftarrow}(u))^4$ near $u=0$ and the same split of the expectation integral as in Proposition~\ref{prop:heavytail} give $\E[F_O(m_n)^4]=O(n^{-p})$. Since $\beta_O>5\beta_D/4$, we have $p>5$, and Assumption~\ref{ass:generator}(A3) follows.
\end{proof}

The shifted exponential case lies just outside the sufficient regime in Proposition~\ref{prop:uniformtype}. Suppose, for example, that $Q_{ii}+a$ are i.i.d. exponential random variables with rate $\lambda_D$ and $Q_{ij}+a$ are i.i.d. exponential random variables with rate $\lambda_O$; both distributions then have the common lower endpoint $-a$. Writing $T:=m_n+a$, we have $T\sim\mathrm{Exp}(n\lambda_D)$ and $F_O(m_n)=1-e^{-\lambda_O T}$. Therefore
\[
\E[F_O(m_n)^4]
=\int_0^\infty (1-e^{-\lambda_O t})^4 n\lambda_D e^{-n\lambda_D t}\,dt.
\]
With the change of variables $v=n\lambda_D t$, this becomes
\[
\int_0^\infty \left(1-e^{-(\lambda_O/(n\lambda_D))v}\right)^4 e^{-v}\,dv.
\]
Set $a_n:=\lambda_O/(n\lambda_D)$. Dividing by $a_n^4$, the normalized integrand is
\[
\left(\frac{1-e^{-a_nv}}{a_n}\right)^4 e^{-v}.
\]
It converges pointwise to $v^4e^{-v}$, and it is dominated by $v^4e^{-v}$ because $1-e^{-z}\le z$ for $z\ge0$. Dominated convergence therefore gives
\[
\E[F_O(m_n)^4]
\sim a_n^4\int_0^\infty v^4e^{-v}\,dv
=24\left(\frac{\lambda_O}{\lambda_D}\right)^4 n^{-4}.
\]
Consequently $n^5\E[F_O(m_n)^4]$ diverges linearly in $n$, so Assumption~\ref{ass:generator}(A3) fails in the case where both the diagonal and off-diagonal laws are shifted exponentials with the same lower endpoint. These laws satisfy the endpoint asymptotic form $F_D(-a+t)\sim \lambda_D t$ and $F_O(-a+t)\sim \lambda_O t$, so $\beta_D=\beta_O=1$; they are therefore not covered by Proposition~\ref{prop:uniformtype}, whose sufficient condition requires $\beta_O>5\beta_D/4$.

These propositions identify \emph{sufficient} regimes for exactness, though they may not be tight thresholds. Once size-five defect components become likely, the first dimension where a DNN-versus-CP separation could arise has been reached, but a relaxation gap need not occur on a typical random five-vertex block.

\subsection{Algorithmic consequences}
The proof of our main theorem is constructive. Given the matrix entries, one may compute $m_n$, build the defect graph by scanning the off-diagonal entries, extract its connected components, and solve the local KKT systems on components of sizes at most four. Since the input has $N=n^2$ entries, this yields an exact algorithm operating in $O(n^2)$ time on the high-probability event of Theorem~\ref{thm:goe-main}; equivalently, the running time is linear in the input size.

\subsection{The size-five frontier}
The theorem shows that components of sizes one through four are harmless for exactness of the local DNN relaxation. Thus the first dimension in which a DNN-versus-CP separation could affect exactness is dimension five, in accordance with the classical separation between $\DNN^5$ and $\CP^5$; see Burer, Anstreicher, and D\"ur~\cite{BurerAnstreicherDur2009}. At the level of first-moment bounds in the GOE regime, five-vertex trees are the dominant connected obstructions: any denser connected graph on five labeled vertices uses at least five edges and therefore contributes at most $O\!\bigl(n^5\E[q_n^5]\bigr)=O\!\bigl((\ln n)^{5/2}/n^5\bigr)$ by Lemma~\ref{lem:moments-goe}, which is negligible compared with the tree contribution $O\!\bigl((\ln n)^2/n^3\bigr)$. Quantifying whether random five-vertex defect components actually generate a localized relaxation gap remains an open problem.

\subsection{Beyond GOE}
The defect-graph viewpoint extends formally beyond GOE. In sample-covariance or Wishart-type models, however, the edge indicators are no longer conditionally independent. The simple tree count must then be replaced by a dependent random-graph estimate; see Bai and Silverstein~\cite{BaiSilverstein2010}. Likewise, banded or spatially dependent ensembles would require a geometric analysis rather than a purely Erd\H{o}s-R\'enyi perspective.

\section*{Acknowledgements}
The author thanks Shuai Li for a careful reading of the manuscript and comments that improved its presentation.

\section*{Declarations}
\textbf{Funding.} No funding was received to assist with the preparation of this manuscript.

\textbf{Competing interests.} The author has no relevant financial or non-financial interests to disclose.

\textbf{Data availability.} Data sharing is not applicable to this article as no datasets were generated or analyzed during the current study.

\end{document}